\documentclass[12pt,reqno]{amsart}
\usepackage{hyphenat}
\usepackage{hyphenat}
\usepackage[a4paper,margin=2.5cm,top=2.5cm,bottom=2.5cm,centering,headheight=3ex,headsep=4ex,vcentering]{geometry}
\usepackage{amssymb,amsfonts,amsmath,amsthm}
\usepackage{orcidlink}
\usepackage[cal=cm,scr=euler]{mathalfa}
\usepackage{microtype}
\usepackage{mlmodern}
\newtheorem{theorem}{Theorem}
\newtheorem{lemma}{Lemma}
\newtheorem{corollary}{Corollary}
\theoremstyle{definition}

\theoremstyle{remark}

\title{Some Comments on Regular Overpartitions modulo $2^k$}

\author[A. M. Alanazi]{Abdulaziz M. Alanazi\,\orcidlink{0000-0002-7909-2539}}
\address[A. M. Alanazi]{Department of Mathematics, Faculty of Sciences, University of Tabuk, P.O.Box 741, Tabuk 71491, Saudi Arabia}
\email{am.alenezi@ut.edu.sa}

\author[M. P. Saikia]{Manjil P. Saikia\,\orcidlink{0000-0002-2997-6731}}
\address[M. P. Saikia]{Mathematical and Physical Sciences division, School of Arts \& Sciences, Ahmedabad University, Navrangpura, Ahmedabad 380009, Gujarat, India}
\email{manjil.saikia@ahduni.edu.in}



\begin{document}

\begin{abstract}
For coprime integers $\ell,\mu\ge 2$, Alanazi, Munagi, and Saikia (2026) studied
$\overline{R}_{\ell,\mu}(n)$, the number of overpartitions of $n$ in which
no part is divisible by $\ell$ or by $\mu$, together with the
single-modulus analogue $\overline{R}_{\ell}(n)$. We record a simple
combinatorial mechanism that determines both
functions modulo every power of $2$ in terms of the number of distinct
part sizes of the underlying ordinary partition. We also deduce a clean characterization of
$\overline{R}_{\ell}(n)$ and $\overline{R}_{\ell,\mu}(n)$ modulo $4$ in
terms of perfect squares.
\end{abstract}

\maketitle

\section{Introduction}

A partition of a positive integer $n$ is a sequence of non-increasing positive integers $\lambda_1 \geq \lambda_2\geq \cdots \geq \lambda_k$ such that $\lambda_1+\cdots +\lambda_k=n$ (and the $\lambda_i$'s are called the parts of the partition. A famous generalization, namely overpartitions were introduced by Corteel and Lovejoy~\cite{CL}: an
\emph{overpartition} of $n$ is a partition of $n$ in which the first occurrence of each distinct part size may or
may not be overlined. Equivalently, an ordinary partition with $r$
distinct part sizes is the underlying partition of exactly $2^{r}$
overpartitions. Following Alanazi and Munagi~\cite{AM} and Alanazi, Munagi, and
Sellers~\cite{AMSe}, an overpartition is \emph{$\ell$-regular} if none of
its part sizes is divisible by $\ell$; we write $\overline{R}_{\ell}(n)$
for the number of $\ell$-regular overpartitions of $n$. More recently,
Alanazi, Munagi, and Saikia~\cite{AMS} studied the biregular refinement:
for coprime $\ell,\mu\ge 2$, let $\overline{R}_{\ell,\mu}(n)$ be the
number of overpartitions of $n$ in which no part size is divisible by
$\ell$ or by $\mu$. They proved a seven-way combinatorial identity and a
number of congruences via generating functions and Radu's algorithm, and
several of those congruences were subsequently extended by elementary
means in~\cite{Paudel} and~\cite{Ghoshal}.

An important contribution in this direction was made by Kim~\cite{Kim}, who observed that the ordinary overpartition function admits the decomposition
\[
\overline{p}(n)=\sum_{r\ge0}2^rp_r(n),
\]
where $\overline{p}(n)$ is the number of overpartitions of $n$ and $p_r(n)$ denotes the number of ordinary partitions of $n$ having exactly $r$ distinct part sizes. Reducing this identity modulo $8$, Kim obtained
\[
\overline{p}(n)\equiv 2p_1(n)+4p_2(n)\pmod8,
\]
which, together with explicit evaluations of $p_1(n)$ and $p_2(n)$, yields a complete characterization of the overpartition function modulo $8$.

The main theorem (Theorem \ref{thm:master} below) of the present paper may be viewed as a substantial generalization of this structural decomposition. Instead of unrestricted overpartitions, we consider $\ell$-regular and $(\ell,\mu)$-regular overpartitions. Moreover, rather than restricting attention to modulo $8$, we establish congruences modulo every power of $2$. In particular, we show that for any partition class defined solely by restrictions on part sizes, the corresponding overpartition function modulo $2^k$ depends only on ordinary partitions having fewer than $k$ distinct part sizes. Thus Kim's decomposition appears as the unrestricted special case of our general theorem.
The generating functions of
the two families that we are interested in are
\begin{equation}\label{eq:gf-l}
\sum_{n\ge 0}\overline{R}_{\ell}(n)\,q^{n}
=\prod_{\ell\nmid k}\frac{1+q^{k}}{1-q^{k}}
,
\end{equation}
and
\begin{equation}\label{eq:gf-lmu}
\sum_{n\ge 0}\overline{R}_{\ell,\mu}(n)\,q^{n}
=\prod_{\substack{\ell\nmid k\\ \mu\nmid k}}\frac{1+q^{k}}{1-q^{k}}.
\end{equation}

The paper is organized as follows: in Section \ref{sec:prelim} we give some preliminary results used to prove our main result in Section \ref{sec:main}. We end the paper with some concluding remarks in Section \ref{sec:conc}.

\section{Preliminaries}\label{sec:prelim}

Throughout, $\tau$ is the divisor-counting function, with the
convention $\tau(x)=0$ whenever $x\notin\mathbb{Z}_{\ge1}$, so that
$\tau(n/\ell)=0$ when $\ell\nmid n$. For coprime $\ell,\mu$ we set
\[
\Delta_{\ell,\mu}(n)=\tau(n)-\tau\!\left(\tfrac{n}{\ell}\right)
-\tau\!\left(\tfrac{n}{\mu}\right)+\tau\!\left(\tfrac{n}{\ell\mu}\right).
\]
For integers $a>b\ge1$ we let $S_{a,b}(n)$ be the number of partitions of
$n$ whose set of distinct part sizes is exactly $\{a,b\}$, and for
$a>b>c\ge1$ we let $T_{a,b,c}(n)$ be the number of partitions of $n$
whose set of distinct part sizes is exactly $\{a,b,c\}$; equivalently
\begin{gather*}
S_{a,b}(n)=\#\{(u,v)\in\mathbb{Z}_{\ge1}^{2}:au+bv=n\},\\
T_{a,b,c}(n)=\#\{(u,v,w)\in\mathbb{Z}_{\ge1}^{3}:au+bv+cw=n\}.
\end{gather*}
For $a>b\ (>c)$, we set
\[
\epsilon_{\ell}(a,b)=
\begin{cases}1,&\ell\nmid a,\ \ell\nmid b,\\0,&\text{else,}\end{cases}
\qquad
\epsilon_{\ell,\mu}(a,b)=
\begin{cases}1,&\ell,\mu\nmid a\text{ and }\ell,\mu\nmid b,\\0,&\text{else,}\end{cases}
\]
and $\epsilon_{\ell}(a,b,c)$, $\epsilon_{\ell,\mu}(a,b,c)$ defined
analogously on three part sizes.

For a class $\mathcal{R}$ of ordinary partitions defined by conditions on
part sizes alone (such as $\ell$-regularity or $(\ell,\mu)$-regularity),
let $N_{r}^{\mathcal{R}}(n)$ denote the number of partitions in
$\mathcal{R}$ of $n$ with exactly $r$ distinct part sizes. We have the following lemmas.

\begin{lemma}\label{lem:lift}
Let $\mathcal{R}$ be a class of ordinary partitions defined by conditions
on part sizes alone, and let $\overline{R}(n)$ count the overpartitions
of $n$ whose underlying ordinary partition lies in $\mathcal{R}$. Then
\[
\overline{R}(n)=\sum_{r\ge1}2^{r}N_{r}^{\mathcal{R}}(n),
\qquad\text{and for every }k\ge1,\quad
\overline{R}(n)\equiv\sum_{r=1}^{k-1}2^{r}N_{r}^{\mathcal{R}}(n)
\pmod{2^{k}}.
\]
\end{lemma}

\begin{proof}
Overlining the first occurrence of a part size does not change the
multiset of part sizes, so an overpartition belongs to the class exactly
when its underlying ordinary partition does, and each partition with $r$
distinct part sizes lifts to exactly $2^{r}$ such overpartitions by an
independent binary choice per distinct part size. Summing over
$\mathcal{R}$ gives the first identity. The second identity is trivial.
\end{proof}
Remark.
In the unrestricted case, Lemma~\ref{lem:lift} reduces to the decomposition used by Kim~\cite{Kim} in the study of the ordinary overpartition function modulo $8$. The present lemma extends that principle to arbitrary partition classes defined by conditions on part sizes and to congruences modulo every power of $2$.
The second lemma is not difficult to prove, so we omit the details.

\begin{lemma}\label{lem:div}
Let $n\ge1$. The number of divisors $d\mid n$ with $\ell\nmid d$ is
$\tau(n)-\tau(n/\ell)$. If $\gcd(\ell,\mu)=1$, the number of divisors
$d\mid n$ with $\ell\nmid d$ and $\mu\nmid d$ is $\Delta_{\ell,\mu}(n)$.
\end{lemma}

We note that the functions $N_{r}$ are generating-function
coefficients. Marking each used part size by a variable $x$,
\begin{equation}\label{eq:bivariate}
\sum_{n,r\ge0}N_{r}^{(\ell)}(n)\,x^{r}q^{n}
=\prod_{\ell\nmid k}\left(1+x\,\frac{q^{k}}{1-q^{k}}\right),
\end{equation}
and setting $x=2$ recovers~\eqref{eq:gf-l} (and likewise for the
biregular product). Thus Lemma~\ref{lem:lift} says that working
modulo $2^{k}$ amounts to truncating the polynomial
$\sum_{r}N_{r}(n)x^{r}$ below degree $k$ and evaluating at $x=2$. Here
$N_{r}^{(\ell)}$ is $N_{r}^{\mathcal{R}}$ for $\mathcal{R}$ the
$\ell$-regular partitions, and $N_{r}^{(\ell,\mu)}$ for the biregular
class.

\section{Congruences modulo $2^{k}$}\label{sec:main}

\begin{theorem}\label{thm:master}
For every $k\ge1$ and every $n\ge1$,
\[
\overline{R}_{\ell}(n)\equiv\sum_{r=1}^{k-1}2^{r}N_{r}^{(\ell)}(n)
\pmod{2^{k}},
\qquad
\overline{R}_{\ell,\mu}(n)\equiv\sum_{r=1}^{k-1}2^{r}N_{r}^{(\ell,\mu)}(n)
\pmod{2^{k}}.
\]
Moreover, we have explicitly
$N_{1}^{(\ell)}(n)=\tau(n)-\tau(n/\ell)$ and
$N_{1}^{(\ell,\mu)}(n)=\Delta_{\ell,\mu}(n)$, while
\[
N_{2}^{(\ell)}(n)=\sum_{a>b\ge1}\epsilon_{\ell}(a,b)S_{a,b}(n),\qquad
N_{2}^{(\ell,\mu)}(n)=\sum_{a>b\ge1}\epsilon_{\ell,\mu}(a,b)S_{a,b}(n),
\]
and similarly $N_{3}=\sum_{a>b>c}\epsilon(a,b,c)T_{a,b,c}(n)$ in each
class.
\end{theorem}

\begin{proof}
Both congruences are just Lemma~\ref{lem:lift} applied to the $\ell$-regular
and $(\ell,\mu)$-regular classes. 

For the explicit terms: a partition
with $r=1$ has the form $(d^{\,n/d})$ with $d\mid n$, and it is
$\ell$-regular (resp.\ $(\ell,\mu)$-regular) exactly when $\ell\nmid d$
(resp.\ $\ell\nmid d$ and $\mu\nmid d$), so $N_{1}$ is the corresponding
divisor count from Lemma~\ref{lem:div}. 

A partition with $r=2$ has a
distinct part-size set $\{a,b\}$ with $a>b$, occurring $S_{a,b}(n)$
times, and is regular exactly when the relevant indicator $\epsilon$
equals $1$; summing over all pairs gives $N_{2}$. The case $r=3$ is
identical with triples and $T_{a,b,c}$.
\end{proof}

The case $k=2$ keeps only
$N_{1}$, so we have
\[
\overline{R}_{\ell}(n)\equiv 2\bigl(\tau(n)-\tau(n/\ell)\bigr)\pmod 4,
\qquad
\overline{R}_{\ell,\mu}(n)\equiv 2\,\Delta_{\ell,\mu}(n)\pmod 4.
\]

The case $k=3$ is the natural analogue of Kim's modulo $8$ decomposition for unrestricted overpartitions. Here the two-distinct-part contribution is replaced by its $\ell$-regular (respectively $(\ell,\mu)$-regular) analogue.

The case $k=4$ keeps $N_{1},N_{2},N_{3}$ and produces the modulo $16$
characterizations.

The modulo $4$ case is interesting, because
$\tau(m)$ is odd if and only if $m$ is a perfect square.
\begin{corollary}\label{cor:square-l}
For all $n\ge1$,
\[
\overline{R}_{\ell}(n)\equiv
\begin{cases}
2\pmod 4,&\text{if exactly one of }n,\ n/\ell\text{ is a perfect square},\\
0\pmod 4,&\text{otherwise.}
\end{cases}
\]
In particular $\overline{R}_{\ell}(n)\equiv 0\pmod 4$ for a set of $n$ of
natural density $1$.
\end{corollary}

\begin{proof}
By the $k=2$ case, $\overline{R}_{\ell}(n)\equiv
2\bigl(\tau(n)-\tau(n/\ell)\bigr)\pmod 4$, and $2x\bmod 4$ depends only on
$x\bmod 2$. Since $\tau(m)\equiv[m\text{ is a square}]\pmod 2$\footnote{where we use the Iverson bracket, where [logical statement] is $1$ if the statement is true, else it is $0$.}, we get
$\tau(n)-\tau(n/\ell)\equiv[n\text{ square}]+[n/\ell\text{ square}]
\pmod 2$, which is $1$ exactly when one of the two is a square and the
other is not.

As perfect squares have density $0$, so does the set where
the residue is $2$.
\end{proof}

\begin{corollary}\label{cor:square-lmu}
For all $n\ge1$, $\overline{R}_{\ell,\mu}(n)\equiv 2\pmod 4$ if and only
if an odd number of $n,\ n/\ell,\ n/\mu,\ n/(\ell\mu)$ are perfect
squares, and $\overline{R}_{\ell,\mu}(n)\equiv0\pmod 4$ otherwise.
\end{corollary}

\begin{proof}
Identical to Corollary~\ref{cor:square-l}, applied to
$\overline{R}_{\ell,\mu}(n)\equiv 2\,\Delta_{\ell,\mu}(n)\pmod 4$ with the
four square-indicators of $n,n/\ell,n/\mu,n/(\ell\mu)$.
\end{proof}

Note that Corollary \ref{cor:square-l} was obtained in \cite[Corollary 2]{AMSe}  and Corollary \ref{cor:square-lmu} was obtained in \cite[Theorem 3.4]{AMS} in different but equivalent forms.

We make the modulo $8$ case explicit.

\begin{theorem}\label{thm:mod8}
For all $n\ge1$,
\begin{align*}
\overline{R}_{\ell}(n)
&\equiv 2\bigl(\tau(n)-\tau(n/\ell)\bigr)
+4\sum_{a>b\ge1}\epsilon_{\ell}(a,b)\,S_{a,b}(n)\pmod 8,\\[2pt]
\overline{R}_{\ell,\mu}(n)
&\equiv 2\,\Delta_{\ell,\mu}(n)
+4\sum_{a>b\ge1}\epsilon_{\ell,\mu}(a,b)\,S_{a,b}(n)\pmod 8.
\end{align*}
\end{theorem}

\begin{proof}
Take $k=3$ in Theorem~\ref{thm:master}; only $N_{1}$ and
$N_{2}$ survive modulo $8$, and both are given explicitly.
\end{proof}

\section{Concluding remarks}\label{sec:conc}

\begin{enumerate}
\item Theorem~\ref{thm:master} may be viewed as a generalization of Kim's structural decomposition of the ordinary overpartition function. In the unrestricted case, Kim evaluated the parity of the two-distinct-part contribution and thereby obtained a complete characterization modulo $8$. An analogous evaluation of
\[
N_2^{(\ell)}(n)
\quad\text{and}\quad
N_2^{(\ell,\mu)}(n)
\]
would immediately yield complete modulo $8$ characterizations for regular and biregular overpartitions.
\item Theorem~\ref{thm:mod8} expresses $R_\ell(n)$ and $R_{\ell,\mu}(n)$ modulo $8$ in terms of the number of regular partitions having exactly one and two distinct part sizes. This is the direct analogue of Kim's decomposition for unrestricted overpartitions. To obtain a complete characterization modulo $8$, it therefore remains to determine the parity of the two-distinct-part term
\(
N_2^{(\ell)}(n)
\)
and its biregular analogue
\(
N_2^{(\ell,\mu)}(n).
\)

In the unrestricted case, Kim showed that the parity of $p_2(n)$ admits an explicit arithmetic description, leading to a complete modulo $8$ characterization of $\overline p(n)$. It is therefore natural to ask for analogous descriptions in the regular and biregular settings.
    \item  It would also be interesting to determine the parity of
\[
N_3^{(\ell)}(n)
\quad\text{and}\quad
N_3^{(\ell,\mu)}(n),
\]
which would lead to explicit congruences modulo $16$.
\item Theorem~\ref{thm:master} provides a unified framework for studying congruences modulo powers of $2$. Rather than working directly with overpartitions, it reduces every congruence modulo $2^k$ to the enumeration of ordinary partitions with fewer than $k$ distinct part sizes. This viewpoint extends Kim's approach from unrestricted overpartitions to general classes of restricted overpartitions.
\end{enumerate}

\end{document}